\documentclass[12pt]{amsart}

\usepackage{amsmath, graphicx, cite, enumerate, comment}
\usepackage{amssymb, amsthm, amscd}
\usepackage{latexsym}
\usepackage[T1]{fontenc}	
\usepackage[english]{babel}
\usepackage{amsfonts}
\usepackage{amscd}
\usepackage{cite}
\usepackage{marvosym}
\usepackage{mathrsfs}
\usepackage{amsthm}
\usepackage{verbatim}

\newtheorem{theorem}{Theorem}[section]

\newtheorem{lemma}[theorem]{Lemma}
\newtheorem{proposition}[theorem]{Proposition}

\theoremstyle{definition}
\newtheorem{definition}[theorem]{Definition}

\theoremstyle{remark}
\newtheorem{remark}[theorem]{Remark}

\numberwithin{equation}{section}

\newcommand{\R}{\mathbb{R}}

\newcommand{\g}{\gamma}

\begin{document}
\date{}
\title[Minimal hyperspheres of arbitrarily large Morse index]
{Minimal hyperspheres of arbitrarily large Morse index}
\author{Alessandro Carlotto}
\address{Department of Pure Mathematics \\
                 Imperial College \\
                 London, SW7 2AZ}
\email{a.carlotto@imperial.ac.uk}

\begin{abstract} We show that the Morse index of a closed minimal hypersurface in a four-dimensional Riemannian manifold cannot be bound in terms of the volume and the topological invariants of the hypersurface itself by presenting a method for constructing Riemannian metrics on $S^4$ that admit embedded minimal hyperspheres of uniformly bounded volume and arbitrarily large Morse index.
The phenomena we exhibit are in striking contrast with the three-dimensional compactness results by Choi-Schoen.	
	 \end{abstract}

\maketitle

\section{Introduction}\label{sec:intro}
In 1970, during his plenary address entitled \textsl{Differential Geometry: its past and its future} \cite{Che70} at the International Congress of Mathematicians held in Nice, S. S. Chern asked the following question:

\

\begin{center}
\textsl{Is it true that an embedded, minimal hypersphere inside the Euclidean $n$-sphere is necessarily an equator?}
\end{center}

\

We shall recall here that at the time there were good reasons to believe the answer to this question had to be affirmative for any dimension, since just a few years earlier F. Almgren had proven such a rigidity statement for $n=3$ \cite{Alm66}, which of course extends the $n=2$ case that amounts to a trivial ODE uniqueness argument. It was therefore quite a surprise for the mathematical community when W. Hsiang \cite{Hsi83a} answered Chern's question in the negative for $n=4, 5, 6$ by constructing (in each of those cases) a sequence $\Sigma_{k}$ of embedded, minimal hyperspheres that were not totally geodesic. This was later extended to $n=7, 8, 10, 12, 14$ in \cite{Hsi83b} and to all even dimensions $n\geq 4$ in \cite{HS86}.
While providing a highly unexpected answer to the aforementioned problem, Hsiang's work had the disadvantage of fully relying on an equivariant construction (in the spirit of \cite{HL71}) and hence did not shed any light on the class of minimally embedded hyperspheres of $S^{n}$ for non-round Riemannian metrics or, even more ambitiously, on the structure of the moduli space of those submanifolds. In this article, we shall prove that the exotic phenomena disclosed by Hsiang are not at all peculiar of the round metric, for in fact there exists on $S^{4}$ an overabundance of Riemannian metrics that have minimal hyperspheres of uniformly bounded area and arbitrarily large Morse index. 

In order to state our main results, we need to introduce some notation. Given an integer $q\geq 3$ and $\alpha\in(0,1/2)$, let $\Gamma=\Gamma(q,\alpha)$ be the space of $\mathcal{C}^{q,\alpha}-$Riemannian metrics on $S^{4}$ and let us agree to denote by $\gamma_0\in\Gamma$ the round metric. For $\gamma\in \Gamma$ we shall consider $[\gamma]$ to be the equivalence class of $\gamma$ modulo (pointwise) conformal equivalence and $\Pi:\Gamma\to K$ to be the corresponding projection.

\newpage

\begin{theorem}\label{thm:main}	
	There exists a neighborhood of Riemannian metrics $\mathcal{U}\subset\Gamma(q+1,\alpha)$ of $\gamma_0$ on the four-sphere such that the following statement holds: for any $[\gamma]\in \Pi(\mathcal{U})$ with vanishing Weyl tensor around two antipodal points, one can construct a converging sequence $\left\{\gamma_{k}\right\}\subset \Gamma(q,\alpha)$ and embedded hyperspheres $\left\{M_k\right\}$ with i) $\gamma_k$ conformal to $\gamma$ (namely $\Pi(\gamma_k)=[\gamma]$), ii) $M_k$ minimal in $(S^{4}, \gamma_k)$ and iii) $\lim_{k\to\infty} Ind(M_k)=\infty$. 
\end{theorem}	

For instance, this theorem yields new results even in the case of perturbations of the round metric of $S^4$ which are supported on a given compact domain not containing the north and south poles.

In fact, the same conclusion of Theorem \ref{thm:main} holds true under a \textsl{pointwise} assumption, namely provided we restrict our consideration to those nearly-round metrics whose Riemann curvature tensor vanishes at a couple of antipodal points to sufficiently high order.

\begin{theorem}\label{thm:main2}	
	There exists a neighborhood of Riemannian metrics $\mathcal{U}\subset\Gamma(q+1,\alpha)$ of $\gamma_0$ on the four-sphere such that the following statement holds: for any $\gamma\in \mathcal{U}$ whose curvature tensor coincides satisfies, at two antipodal points, the equations
	\[
	Riem^\gamma=Riem^{\gamma_0}, \ \ \nabla_{\gamma} Riem^{\gamma} =\ldots=\nabla^{(q-3)}_{\gamma}Riem^{\gamma}=0
	\]
	one can construct a converging sequence $\left\{\gamma_{k}\right\}\subset \Gamma(q,\alpha)$ and embedded hyperspheres $\left\{M_k\right\}$ with i) $\gamma_k$ conformal to $\gamma$, ii) $M_k$ minimal in $(S^{4}, \gamma_k)$ and iii) $\lim_{k\to\infty} Ind(M_k)=\infty$. 
\end{theorem}	

Of course, in the previous statements $Ind(M)$ stands for the Morse index of the minimal submanifold $M$, that is the number of negative eigenvalues of the Jacobi operator $J_{M}$ given by
\[
J_{M}u=\Delta_{M}u+\left(|A|^{2}+Ric(\nu,\nu)\right)
\]
where $A$ is the second fundamental form of $M$ in the ambient manifold $(N,\gamma)$ under consideration, $Ric(\cdot,\cdot\cdot)$ is the Ricci curvature tensor of such manifold and $\nu$ in the unit normal of $M$ inside $N$.
The problem of explicitly computing, or even just getting effective estimates on the Morse index of a given minimal submanifold is in general very delicate and has been tackled only in very few well-known cases. We shall start here by recalling that the equatorial hyperspheres in $(S^{n},\gamma_0)$ have Morse index equal to 1, instead when the reference metric $\gamma$ is not round but has positive Ricci curvature the index is only known to be strictly positive. When $n=3$ Eijiri and Micallef \cite{EM08} gave a remarkable, general upper bound on the Morse index in terms of the area of a closed minimal surface $M$ in a compact 3-manifold $(N,\gamma)$: when $M\simeq S^{2}$ this takes the simple form
\[
Ind(M)\leq C(N)\mathscr{H}^{2}(M)
\]
where $C(N)$ is a constant depending on the second fundamental form of an isometric embedding of $(N,\gamma)$ into Euclidean space. On the other hand, Choi and Schoen \cite{CS85} had proven that the area of a minimal embedding can be controlled by means of the genus so that in the end one achieves, for minimal spheres in 3-manifolds the bound
\[
Ind(M)\leq C(N)\frac{32\pi}{\kappa}\left(\frac{1}{|\pi_{1}(N)|}\right)
\]
provided $Ric_{\gamma}\geq\kappa>0$.
Obviously, this inequality ensures that the phenomena described in the statements of Theorem \ref{thm:main} and Theorem \ref{thm:main2} cannot possibly occur when $n=3$. In this respect, we shall remark that (by our very construction) all sequences $\left\{\gamma_k\right\}$ as in those statements are  also contained in a suitably small neighborhood of $\gamma_0$, so that a uniform positive lower bound on the Ricci curvature is guaranteed.
In fact, Cheng and Tysk \cite{CT94} could adapt the heat kernel technique of Li-Yau to prove an upper bound for the number of nonpositive eigenvalues of a Schr\"odinger operator of the form $L=\Delta_M+V$ whenever $M$ is a minimally immersed submanifold of dimension at least 3 in a closed manifold $(N,\gamma)$: this reads
\[
\# \left\{\lambda_j \ : \ \Delta_{M}u+Vu=-\lambda_j u \ \textrm{and} \ \lambda_j\leq 0\right\}\leq C(m,N)\int_{M}\left(\textrm{max}(V,1)\right)^{m/2}\,d\mathscr{H}^{m}.
\]
However, this constraint becomes \textsl{vacuous} when referred to our construction, for one can easily check that the sequences $\left\{M_k\right\}$ of minimal hyperspheres we construct satisfy
\[
\lim_{k\to\infty}\int_{M_k}|A|^{3}\,d\mathscr{H}^{3}=+\infty.
\]

\

We shall now compare the content of Theorem \ref{thm:main} and Theorem \ref{thm:main2} with the other existence results for closed minimal hypersurfaces in a four-manifold:
\begin{itemize}
\item{the general min-max theory due to Almgren and Pitts, see the monograph \cite{Pit81}, ensures the existence of one closed smooth, embedded minimal hypersurface; yet, differently from the $n=3$ case this is not known to be a hypersphere if the ambient $N$ is diffeomorphic to the standard $S^{4}$ (namely there is no four-dimensional analogue of the theorem by F. Smith \cite{Smi82}), furthemore the Morse index is quite delicate to be controlled and, in any case, is expected to equal 1;}	
\item{the pertubative methods developed by B. White in \cite{Whi91} imply the existence of at least 5 minimal hyperspheres for any nearly-round metric $\gamma$ on $S^{4}$, and by the very method they are constructed they have index bounded above by 5;}
\item{the recent results obtained by F. Marques and A. Neves in \cite{MN13} and based on min-max schemes with high-dimensional parameter spaces guarantee the existence in any four-manifold $(N,\gamma)$ of positive Ricci curvature of infinitely many closed, embedded minimal hypersurfaces. However, neither the topological complexity nor the Morse index of those elements are, at the moment, reasonably well-understood. In this respect, when discussing the open problems related to the min-max hypersurfaces obtained by considering their $p$-dimensional sweepouts $\mathcal{P}_{p}$ Marques and Neves make the following statement, Section 9 in \cite{MN13} : `One could naively expect that under generic conditions they should have index $p$, multiplicity one and their volumes converge to infinity.'}	
\end{itemize}	 
As a result, the construction we present is, to the best of our knowledge, the first exhibition of codimension-one embedded, minimal submanifolds with fixed topology, bounded volume and arbitrarily large Morse index.\footnote{Of course, \textsl{a posteriori} the Hsiang desingularizations also give such an example, even though Hsiang did not prove any result about the Morse index of those embedded minimal hyperpheres (and did not have the tools to prove the divergence thereof).}

\medskip

Let us now briefly describe the conceptual scheme of the proof of our main theorems and, correspondingly, the structure of this article. The sequences of minimal hyperspheres in $(S^{4},\gamma_0)$ constructed by Hsiang in \cite{Hsi83a} can be seen to converge, in the sense of varifolds, to a singular limit $M$ which we shall call \textsl{Clifford football}: that is a 3-dimensional minimally embedded subvariety of the four-sphere that is homeomorphic to the suspension $T^{2}\times [0,1]/\sim$ (where $T^{2}$ is the 2-torus and $\sim$ is the equivalence relations that pinches the two boundary components to points) and has two conical singularities located at antipodal points on $S^{4}$. The reason for the choice of such a name is that the blow-up of $M$ at each of those singularities is the cone over the Clifford torus $S^{1}\left(1/\sqrt{2}\right)\times S^{1}\left(1/\sqrt{2}\right)\subset S^{3}\subset \mathbb{R}^{4}$.
Now, the basic idea of our construction is to \textsl{first} deform the Clifford football as we vary the background Riemannian metric in a neighbourhood of $\gamma_0$ and \textsl{second} desingularize the corresponding perturbed Clifford footballs. 
Concerning the first step, our precise statement is as follows:

\begin{theorem}\label{pro:deform}
	Let $q\geq 3, \alpha,\alpha'\in (0,1/2)$ and $\boldsymbol{\beta}=(\beta_1,\beta_2)$ for some $\beta_1=\beta_2=\beta>1+\alpha'+q$. Then there exist bounded neighborhoods $\mathcal{U}\subset\Gamma(q+1,\alpha),\mathcal{V}\subset\Gamma(q,\alpha')$ of the round metric $\gamma_0$ on $S^4$, $\mathcal{W}\subset\mathcal{W}^{q+2,2}_{\boldsymbol{\beta}}$ of the function identically equal to zero and $\mathcal{C}^{1}$ maps $\Xi:\mathcal{U}\to\mathcal{V}$ and $\Omega:\mathcal{U}\to\mathcal{W}$ such that i) for all $\gamma\in\mathcal{U}$ the metric $\Xi(\gamma)$ is conformal to $\gamma$ and ii) the normal graph defined (over the Clifford football) by $\Omega(\gamma)$ is a singular minimal submanifold in $(S^{4},\Xi(\gamma))$. 	
\end{theorem}

This can be considered a perturbative result relative to a geometric problem for which not only a direct application of the Implicit Function Theorem, but also any sort of Lyapunov-Schmidt reduction is ineffective (as will be further explained in Section \ref{sec:concl}).
Indeed, we exploit the freedom on the conformal factor, namely the fact that we are working inside a conformal class rather than with a fixed background metric, in order to overcome the obstructions related to i) the action of global isometries on $(S^{4},\gamma_0)$ and ii) the presence of \textsl{regularizing modes}, associated to desingularizations of the Clifford football at each of its singularities.

Roughly speaking, we can then produce (inside the conformal class $[\gamma]$) a family of minimal embedded desingularizations of $M_{\Omega(\gamma)}$ which converge to a limit that has nonplanar tangent cones at two antipodal points.
At that stage, the conclusion comes, arguing by contradiction, by means of the recent compactness theorem by B. Sharp \cite{Sha15}. For if there were a uniform upper bound on the Morse indices of the elements of $\left\{M_k\right\}$ then there should be a subsequence converging to a \textsl{smooth} embedded minimal hypersurface, which is not the case.

The structure of the article is as follows: in Section \ref{sec:prelim} we present some background material concerning submanifolds with conical singularities, weighted functional spaces and then specialize our discussion to the Clifford football and its Hsiang and Alencar regularization, in Section \ref{sec:deform} we deform the Clifford football in order to obtain singular minimal submanifolds for nearly-round metrics and in Section \ref{sec:desing} we desingularize such elements in order to obtain smooth minimal hyperspheres. Finally, we present in Section \ref{sec:concl} a series of remarks about our construction, variations thereof and related open problems.

\medskip

\textsl{Acknowledgments}. The author wishes to express his deepest gratitude to Prof. Andr\'e Neves for suggesting the problems this project arose from and for a number of enlightening conversations. He would also like to thank Mark Haskins and Andrea Malchiodi for several useful discussions and for their interest in this work. Furthermore, he is indebted to Ben Sharp for clarifying some aspects concerning the applicability of his recent compactness theorem. During the preparation of this article, the author was supported by Prof. Neves European Research Council Start Grant.

\section{Preliminaries and recollections}\label{sec:prelim}

This article concerns the deformation and desingularization of minimal submanifolds with isolated singularities, therefore let us start by defining this category and describing the functional set-up we will consider in the sequel.

\subsection{Manifolds with isolated conical singularities}\label{subs:cate}

\begin{definition}\label{def:singman}
	Given an integer $m\geq 1$ we define an $m$-dimensional manifold with isolated singularities (of class $\mathcal{C}^{k,\alpha}$ or, respectively, $\mathcal{C}^{\infty}$)  to be a triple $(M,S,d)$ where $S$ is a finite (yet possibly empty) set $\left\{p_{1},\ldots,p_{e}\right\}\subset M$ such that $(M,d)$ is a compact metric space and the following conditions hold:
	\begin{enumerate}
		\item{the set $\dot{M}:=M\setminus \left\{p_{1},\ldots, p_{e}\right\}$ is an open manifold of class $\mathcal{C}^{k,\alpha}$ (resp. $\mathcal{C}^{\infty}$); }
		\item{there exists a compact set $C\subset M$ such that $M\setminus C=\bigsqcup_{i=1}^{e}E_i$ and for each value of the index $i$ there exists a smooth, closed connected $(m-1)$-manifold $P_{i}$ such that $\phi_i:(0,1]\times P_i\to \overline{E_i}$ is a diffeomorphism (of the appropriate level of regularity, as above)}
		\item{there exists a $\mathcal{C}^{k,\alpha}$ (resp. $\mathcal{C}^{\infty}$) Riemannian metric $\g$ on $\dot{M}$ that induces the distance $d$ and furthermore for positive constants $\nu_1,\ldots,\nu_e$
		\[
		|\tilde{\nabla}^{j}(\phi^{\ast}_{i}\g-\tilde{\g}_{i})|_{\tilde{\g}_i}=O(r^{\nu_i-j}) \ \forall \ 0\leq j\leq k, \ \
	     [\tilde{\nabla}^{k}(\phi^{\ast}_{i}\g-\tilde{\g}_{i})]^{\tilde{\g_i}}_{\alpha}=O(r^{\nu_i-j-\alpha})
		\]	
		(resp. $|\tilde{\nabla}^{j}(\phi^{\ast}_{i}\g-\tilde{\g}_{i})|_{\tilde{\g_i}}=O(r^{\nu_i-j})$ for all $j\geq 0$), where $\tilde{\g}_{i}=dr^{2}+r^{2}\g'_{i}$ for coordinates $(\theta,r)\in P_i\times (0,1]$ and $\g'_i$ a Riemannian metric on $P_i$.}
	\end{enumerate}	
	Here $k\geq 2$ is an integer, $\alpha\in (0,1)$.
\end{definition}	

It is straightforward to check that each $P_{i}$ is uniquely determined and hence there is a well-defined notion of singular model at each singular point. Notice that by allowing the set of singular points to be empty we allow regular manifolds to be regarded as (exceptional) manifolds with isolated singularities, which is just convenient in a number of situations. 

When defining weighted Sobolev and H\"older spaces, we will make use of a \textsl{radius function}.

\begin{definition}\label{def:radius}
Given a manifold with conical singularities $(M,S,d)$ as per Definition \ref{def:singman} given above, we will say that $\rho:M\to (0,\infty)$ is a radius function if $\rho=d(p_i,\cdot)$ on $E_i$ for any $i=1,\ldots, e$.
\end{definition}

Given a multi-index $\boldsymbol{\beta}=(\beta_1,\ldots,\beta_e)\in\mathbb{R}^{e}$, we shall now define the functional spaces we need. To that aim, let us agree to denote by $\rho^{\boldsymbol{\beta}}$ a positive function that equals $\rho^{\beta_i}$ along the end $E_{i}\subset M$.

\begin{definition}
Given a manifold with conical singularities $(M,S,d)$ and a multi-index $\boldsymbol{\beta}\in\mathbb{R}^{e}$, we let:
\begin{enumerate}
\item{	
	 $\mathcal{W}^{k,p}_{\boldsymbol{\beta}}(M)$ to be the Banach space completion of $\mathcal{C}^{\infty}(\dot{M})$ with respect to the norm
\[
\left\|u\right\|_{\mathcal{W}^{k,p}_{\boldsymbol{\beta}}}:=\left(\sum_{j=0}^{k}\int_{M}|\rho^{(-\boldsymbol{\beta}+j)}\nabla^{j}u|^{p}\rho^{-m}\,d\mu_{\gamma}\right)^{1/p}.
\]	
When $k=2$, we shall agree to use the notation $\mathcal{H}^{k}_{\boldsymbol{\beta}}(M)$ in lieu of $\mathcal{W}^{k,p}_{\boldsymbol{\beta}}(M)$;}
\item{$\mathcal{C}^{k,\alpha}_{\boldsymbol{\beta}}(M)$ to be the Banach space completion of $\mathcal{C}^{\infty}(\dot{M})$ with respect to the norm
\[
\left\|u\right\|_{\mathcal{C}^{k,\alpha}_{\boldsymbol{\beta}}}:=\sum_{j=0}^{k}\sup_{\dot{M}}\rho^{-\boldsymbol{\beta}+j}|\nabla^{j}u| + \sup_{x\neq y \in \dot{M}}\frac{|\rho^{-\boldsymbol{\beta}+k}(x)\nabla^{k}u(x)-\rho^{-\boldsymbol{\beta}+k}(y)\nabla^{k}u(y)|}{d(x,y)^{\alpha}}.
\]	
	 }
\end{enumerate}	
\end{definition}	

Some fundamental facts about Analysis on manifolds with conical singularities or, more generally, on \textsl{conifolds} have been studied in detail and collected in \cite{Pac13}. For our purposes, we shall state the following version of the Sobolev embedding theorem.

\begin{theorem}\label{thm:Sob}
Let $(M,S,d)$ be a manifold with conical singularities. Assume $k\in\mathbb{N}$, $l\in\mathbb{N}^{\ast}$ and $p\geq 1$. Given a multi-index $\boldsymbol{\beta}$ for all $\boldsymbol{\beta}'\leq\boldsymbol{\beta}$ the following statements hold:
\begin{enumerate}
\item{If $lp<m$ then there exists a continuous embedding $\mathcal{W}^{k+l,p}_{\boldsymbol{\beta}}(M)\hookrightarrow \mathcal{W}^{k,p^{\ast}_l}_{\boldsymbol{\beta}'}(M)$;}	
\item{If $lp=m$ then, for all $q\in [p,\infty)$, there exists a continuous embedding $\mathcal{W}^{k+l,p}_{\boldsymbol{\beta}}(M)\hookrightarrow \mathcal{W}^{k,q}_{\boldsymbol{\beta}'}(M)$;}
\item{If $lp>m$ then, for all $\alpha\in [0,\min\left\{1,l-m/p\right\}]$, there exists a continuous embedding $\mathcal{W}^{k+l,p}_{\boldsymbol{\beta}}(M)\hookrightarrow \mathcal{C}^{k,\alpha}_{\boldsymbol{\beta}'}(M)$. }	
\end{enumerate}		
	Here we have denoted the Sobolev-dual exponent of $p$ by $p^{\ast}_{l}$, namely $p^{\ast}_{l}=\frac{mp}{m-lp}$.
\end{theorem}

\subsection{Minimal submanifolds with isolated conical singularities}\label{subs:subm}

In this work we shall be interested in those manifolds with isolated conical singularities $(M,S,d)$ for which $M$ is a subset of a Riemannian manifold $(N,\gamma)$ and the function $d:M\times M\to \mathbb{R}$ is the restriction to $M\times M$ of the ambient distance determined by $\gamma$. This implies that (with the notations of Definition \ref{def:singman}) the Riemannian metric on $\dot{M}$ is obtained by restriction of $\gamma$ and similarly for all derived structures, starting with the Levi-Civita connection. 

Let us then assume, from now onwards, that $m\geq 2$ so that the singularities have codimension at least two in $M$. A simple argument, based on removing small geodesic balls around the singularities and integrating by parts, gives the following characterization of \textsl{minimality} in our category.

\begin{lemma}\label{lem:first}
	Let $(N,\gamma)$ be a Riemannian manifold and let $(M,S,d)$ be a submanifold with isolated conical singularities (in the sense of \ref{def:singman}). Given a differentiable one-parameter family of diffeomorphisms of $N$, say $\phi_{t}:N\to N$ (with $\phi_{0}=0$), then
	\[
	\left[\frac{d}{dt}\mathscr{H}^{m}((\phi_{t})_{\#}M)\right]_{t=0}=-\int_{\dot{M}}\gamma\left(X,\vec{H}\right)\,d\mathscr{H}^{m}
	\] 	
	where we denote $X=\left(\frac{d\phi_{t}}{dt}\right)_{t=0}$ the deformation vector field. In particular, $M$ is stationary if and only if the mean curvature $\vec{H}$ vanishes along the regular part of $M$, so if and only if it is a singular minimal submanifold of $N$.
\end{lemma}

The following remark ensures that we could \textsl{equivalently} build up our theory in a much weaker setting, namely that of stationary (integer rectifiable) varifolds (we are adopting the terminology of \cite{Sim83}).

\begin{remark} (Varifold perspective)
	\begin{theorem}(see \cite{Sim83b} and \cite{Sim85}, Theorem 5.7)
		Let $(N,\gamma)$ be a Riemannian manifold and let $\textbf{V}$ be integer rectifiable, $m$-dimensional varifold such that $\overline{\textrm{spt}(\textbf{V})}\setminus \textrm{spt}(\textbf{V})$ consists of a finite set $S$. Suppose that at each of the singularities $p_{i}, \ i=1,2,\ldots, e$ there exists a tangent cone $T_{i}$ which is regular and has multiplicity one. Then each $T_{i}$ is the unique tangent cone to $\textbf{V}$ at $p_{i}$ and moreover  there exists $r_{0}>0$ such that $\textbf{V}\cap B_{r_{0}}(p_{i})$ is the graph over $T_{i}$ of a $\mathcal{C}^{2}$ function $h_{i}: T\cap B_{r}(p_{i})\to\underline{E}$ satisfying the estimates $|r^{-1}h_{i}(r\omega)|+|\nabla h(r\omega)|_{\gamma}\to 0$ as $r\to 0$.
		Here $\underline{E}$ stands for the normal bundle to $T_{i}$ in the tangent space $T_{p_{i}}N$ and $r,\omega$ are polar coordinates on $T_{i}$ associated to geodesic normal coordinates for $N$ at the point $p_{i}$.
	\end{theorem}	
\end{remark}	

Let us now restrict our attention to the codimension one case, namely when $\dim(N)=\dim(M)+1\geq 3$. If $M\hookrightarrow (N,\gamma)$ is a \textsl{minimal} submanifold with isolated conical singularities, then we can locally (and globally, whenever $M$ is two-sided) describe the mean curvature vector $\vec{H}$ on the regular part $\dot{M}$ as $\vec{H}=H\nu$ and we shall adopt this convention without futher remarks. Analogously to Lemma \ref{lem:first}, it is an easy exercise to prove the following statement concerning the second variation of the $m$-dimensional area functional.

\begin{lemma}\label{lem:second}
	Let $(N,\gamma)$ be a Riemannian manifold as above and let $(M,S,d)$ be a submanifold with isolated conical singularities (in the sense of Definition \ref{def:singman}). Given a differentiable one-parameter family of diffeomorphisms of $N$, say $\phi_{t}:N\to N$ (with $\phi_{0}=0$ and $\phi_t=id$ outside a compact subset of $M$), then
	\[
	\left[\frac{d^2}{dt^2}\mathscr{H}^{m}((\phi_{t})_{\#}M)\right]_{t=0}=-\int_{\dot{M}}u J_M u\,d\mathscr{H}^{m}
	\] 	
	where $J_M u=\Delta_M u+(|A|^{2}+Ric(\nu,\nu))u$ and $u=\gamma(X,\nu)$ for $X=\left(\frac{d\phi_{t}}{dt}\right)_{t=0}$. 
\end{lemma}	

Of course, a standard approximation argument ensures the validity of such a conclusion whenever $u\in \mathcal{W}^{1,2}_{\frac{2-m}{2}}$.

\subsection{The Clifford football} We let $x=(x_{1},x_{2},x_{3},x_{4},x_{5})$ be Euclidean coordinates on $\mathbb{R}^{5}$ and $S^{4}\hookrightarrow\mathbb{R}^{5}$ be the unit sphere. If $G=O(2)\times O(2)$, one can consider the group action which is gotten by restriction to $S^{4}$ of the standard representation $\rho_G: G\to\mathbb{R}^{5}$ given by $\rho_G=\rho_{2}\oplus \rho_{2}'\oplus 1$. It is well-known that the associated orbit space, namely the quotient $S^{4}/G$ is geometrically a spherical lume that can be described in terms of planar polar coordinates as 
\[ S^{4}/\left(O(2)\times O(2)\right)=\left\{\left(r,\omega\right) | \ 0\leq r\leq\pi, 0\leq\omega\leq\pi/2\right\}
\]
and has an induced orbital distance metric of the form $dr^{2}+\sin^{2}r d\omega^{2}$. 

Throughout this article, we will denote by $M$ the preimage of the $\omega-$bisector, and namely of the set $\left\{\omega=\pi/4\right\}$ by the quotient map $\pi:S^{4}\to S^{4}/G$: it is then well-known \cite{Hsi83a} that $M$ is a three-dimensional, singular minimal hypersurface of $S^{4}$ with two isolated minimal singularities at the north and south pole of such ambient sphere. It is easily seen that the regular horizontal sections of $M$ (that are the intersections $M\cap \left\{x_{5}=\lambda\right\}$ for $\lambda\in (-1,1)$) are isometric to (suitably rescaled) Clifford tori and that, correspondingly, the blow-up of $M$ at both the north and the south pole of $S^{4}$ is the cone $C$ over the unit Clifford torus $T_{Clifford}^{2}\hookrightarrow S^{3}\hookrightarrow \left\{x_{5}=\pm 1\right\}\hookrightarrow\mathbb{R}^{5}$. Because of these remarks, we will call $M$ the \textsl{Clifford football} and $C$ the (unit) Clifford cone.

The regular part $\dot{M}$ of $M$ can be parametrized by means of four charts $F_{\pm\pm}: D_{\pm\pm}\to\mathbb{R}^{5}$ for $D_{\pm\pm}=I_{\pm}\times I_{\pm}\times (0,\pi)$ and $I_{+}=(-\pi,\pi), I_{-}=(0,2\pi)$ that are gotten by restriction of the covering map $F:\mathbb{R}^{2}\times\left(0,\pi\right)\to\mathbb{R}^{5}$ given by
\[
F(\phi,\psi,\theta)=\left(\frac{\sin\theta}{\sqrt{2}}\cos\phi, \frac{\sin\theta}{\sqrt{2}}\sin\phi, \frac{\sin\theta}{\sqrt{2}}\cos\psi, \frac{\sin\theta}{\sqrt{2}}\sin\psi, \cos\theta\right).
\]

Such parametrization determines tangent vectors
\[
\frac{\partial F}{\partial\phi}= F_{\ast}\left(\frac{\partial}{\partial\phi}\right)=\left(-\frac{\sin\theta}{\sqrt{2}}\sin\phi,\frac{\sin\theta}{\sqrt{2}}\cos\phi, 0, 0, 0\right) 
\]
\[
\frac{\partial F}{\partial\psi}= F_{\ast}\left(\frac{\partial}{\partial\psi}\right)=\left(0, 0, -\frac{\sin\theta}{\sqrt{2}}\sin\psi,\frac{\sin\theta}{\sqrt{2}}\cos\psi, 0\right) 
\]
\[
\frac{\partial F}{\partial\theta}= F_{\ast}\left(\frac{\partial}{\partial\theta}\right)=\left(\frac{\cos\theta}{\sqrt{2}}\cos\phi,\frac{\cos\theta}{\sqrt{2}}\sin\phi,\frac{\cos\theta}{\sqrt{2}}\cos\psi,\frac{\cos\theta}{\sqrt{2}}\sin\psi, -\sin\theta\right) 
\]
which are pairwise orthogonal, and hence can be normalized to give the unit frame:
\[
\tau_{1}=\left(-\sin\phi,\cos\phi, 0, 0, 0\right)
\]
\[
\tau_{2}=\left(0,0,-\sin\psi,\cos\psi, 0\right)
\]
\[
\tau_{3}=\left(\frac{\cos\theta}{\sqrt{2}}\cos\phi,\frac{\cos\theta}{\sqrt{2}}\sin\phi,\frac{\cos\theta}{\sqrt{2}}\cos\psi,\frac{\cos\theta}{\sqrt{2}}\sin\psi, -\sin\theta\right). 
\]

Throughout this section, we let $\nu:\dot{M}\to\mathbb{R}^{5}$ the Gauss map of $\dot{M}\hookrightarrow S^{4}$, which we will conveniently consider taking values in $\mathbb{R}^{5}$. Furthermore, we denote by $A$ the second fundamental form of $\dot{M}$. 

\subsection{The Jacobi operator of the Clifford football}\label{subs:Jacobi}

In this subsection we compute the Jacobi operator of the Clifford football. If $u\in\mathcal{C}^{2}(M)$, we shall denote here, in order to avoid ambiguities, $\overline{u}=u\circ F$.

\begin{lemma}
(Notations as above). The Jacobi operator of the Clifford football is given by
\[J_M u=\frac{2}{\sin^{2}\theta}\frac{\partial^{2}\overline{u}}{\partial\phi^{2}}+\frac{2}{\sin^{2}\theta}\frac{\partial^{2}\overline{u}}{\partial\psi^{2}}+\frac{\partial^{2}\overline{u}}{\partial\theta^{2}}+(2\cot\theta)\frac{\partial \overline{u}}{\partial\theta}+\left(3+\frac{2}{\sin^{2}\theta}\right)\overline{u}\].
\end{lemma}

\begin{proof}

As we recalled above, we know that  $J_{M}u=\Delta_{M}u+\left(Ric(\nu,\nu)+|A|^{2}\right)$ so all we need to do is to compute the three summands explicitly.
First of all, since patently $\left\{\tau_{1},\tau_{2},\tau_{3},\nu\right\}$ is a positive orthonormal frame of $\mathbb{R}^{4}$ we have (referring the indices to that basis)
\[ \textit{Ric}(\nu,\nu)=\sum_{1\leq i\leq 4} R_{i4i4} = \sum_{1\leq i\leq 4}\left(g_{ii}g_{44}-g_{i4}g_{i4}\right)=3. 
\]
In order to compute the second fundamental form of $\dot{M}$ we observe that
\[ \nu=\left(\frac{1}{\sqrt{2}}\cos\phi, \frac{1}{\sqrt{2}}\sin\phi, -\frac{1}{\sqrt{2}}\cos\psi,-\frac{1}{\sqrt{2}}\sin\psi, 0\right)
\]
and hence, if we denote by $D$ the covariant derivative induced by the flat metric on $\mathbb{R}^{5}$ we get
\[
D_{\tau_{1}}\nu=\frac{1}{\sin\theta}\left(-\sin\phi,\cos\phi,0,0,0\right), \
D_{\tau_{2}}\nu=\frac{1}{\sin\theta}\left(0,0,\sin\phi,-\cos\phi,0\right), \
D_{\tau_{3}}\nu=0.
\]
At that stage, by projecting onto the tangent space of $\dot{M}$ at the point in question we obtain that the only non-zero terms of the matrix representing the second fundamental form $A$ with respect to the frame are
\[
A(\tau_{1},\tau_{1})=\frac{1}{\sin\theta}, \ \textrm{and} \ \  A(\tau_{2},\tau_{2})=-\frac{1}{\sin\theta}, 
\]
so that finally
\[
|A|^{2}=\frac{2}{\sin^{2}\theta}.
\]

As a third and final step, let us compute the Laplace-Beltrami operator. Making use, once again, of the frame $\left\{\tau_{1},\tau_{2},\tau_{3}\right\}$ defined above we have that $\Delta_{M}u=\sum_{1\leq i\leq 3}\nabla_{\tau_{i}}\nabla_{\tau_{i}}u-\nabla_{\nabla_{\tau_{i}}\tau_{i}}u$ so that clearly the first summand equals
\[ \sum_{1\leq i\leq 3}\nabla_{\tau_{i}}\nabla_{\tau_{i}}u = \frac{2}{\sin^{2}\theta}\frac{\partial^{2}\overline{u}}{\partial\phi^{2}}+\frac{2}{\sin^{2}\theta}\frac{\partial^{2}\overline{u}}{\partial\psi^{2}}+\frac{\partial^{2}\overline{u}}{\partial\theta^{2}}
\]
where we have used the convenient notation $\overline{u}(\phi,\psi,\theta)=u(F(\phi,\psi,\theta))$.
Concerning the torsion terms we get
\[
D_{\tau_{1}}\tau_{1}=-\frac{\sqrt{2}}{\sin\theta}\left(\cos\phi,\sin\phi,0,0,0\right), \
D_{\tau_{2}}\tau_{2}=-\frac{\sqrt{2}}{\sin\theta}\left(0,0,\cos\psi,\sin\psi,0\right)
\]
\[
D_{\tau_{3}}\tau_{3}=\left(-\frac{\sin\theta}{\sqrt{2}}\cos\phi, -\frac{\sin\theta}{\sqrt{2}}\sin\phi,-\frac{\sin\theta}{\sqrt{2}}\cos\psi, -\frac{\sin\theta}{\sqrt{2}}\sin\psi, -\cos\theta \right)
\]
and hence by projecting we find at once
\[
\nabla_{\tau_{1}}\tau_{1}=\nabla_{\tau_{2}}\tau_{2}=-(\cot\theta) \tau_{3}, \ \ \nabla_{\tau_{3}}\tau_{3}=0.
\]
Therefore, putting together the previous two equations, we conclude that
\[
\Delta_M u=\frac{2}{\sin^{2}\theta}\frac{\partial^{2}\overline{u}}{\partial\phi^{2}}+\frac{2}{\sin^{2}\theta}\frac{\partial^{2}\overline{u}}{\partial\psi^{2}}+\frac{\partial^{2}\overline{u}}{\partial\theta^{2}}+(2\cot\theta)\frac{\partial \overline{u}}{\partial\theta}
\]
and hence the claim follows at once.
\end{proof}

\subsection{Alencar and Hsiang desingularizations}\label{subs:reg}

We shall devote the first part of this subsection to the description of the desingularizations of the Clifford cone studied by Alencar in \cite{Ale93}. Following an approach that had already been successfully employed in \cite{BdGG69} in order to prove the area-minimizing property of Simons' cones, Alencar considered the class of minimal hypersurfaces in $\mathbb{R}^{m}\times\mathbb{R}^{m}$ that are invariant under the action of the group $G_m=O(m)\times O(m)$. The corresponding orbit space, in this case, is the first quadrant $\left\{(x,y)\in\mathbb{R}^{2} \ :\ x\geq 0, y\geq 0 \right\}$ and minimality in $\mathbb{R}^{2m}$ for the preimage $\pi_m^{-1} (\textrm{spt}(\sigma))$ corresponds to the requirement that the curve $\sigma(s)=(x(s),y(s))$ satisfies the second-order differential equation
\begin{equation}\label{eq:Ale}
x'(s)y''(s)-x''(s)y'(s)=(m-1)[(x'(s))^{2}+(y'(s))^{2}]\left(\frac{x'(s)}{y(s)}-\frac{y'(s)}{x(s)}\right).
\end{equation}
There are three different situations that may occur:
\begin{enumerate}
\item{the generating curve intersects perperdicularly one of the semi-axes of the orbit space;}	
\item{the generating curve does not intersect the boundary of the orbit space;}
\item{the generating curve passes through the origin of the orbit space.}	
\end{enumerate}	

The third case is well-understood and corresponds, in our setting, to the Clifford cone.

\begin{theorem}(Theorem 4.1 in \cite{Ale93})
Let $M^{2m-1}, \ m\geq 2$ be a minimal hypersurface of $\mathbb{R}^{2m}$ that is invariant under the action of $G_m$ and passes through the origin of $\mathbb{R}^{2m}$. Then $M^{2m-1}$ is (modulo an ambient isometry) the minimal	quadratic cone
\[
C_m=\left\{(X,Y)\in\mathbb{R}^{m}\times\mathbb{R}^{m} \ : \ |X|=|Y| \right\}.
\]
\end{theorem}	

The other two cases are fully classified when $m=2, 3$ which is enough for our purposes as we are dealing with the $m=2$ case.

\begin{theorem}(Theorem 1.1 in \cite{Ale93})\label{thm:Ale}
Let $M^{2m-1}, \ m=2, 3$ be a complete minimal hypersurface in $\mathbb{R}^{2m}\setminus\left\{0\right\}$ that is invariant under the action of $G$. Then:
\begin{enumerate}	
\item[(a)]{\underline{either} $M^{2m-1}$ is embedded and has the topological type of $\mathbb{R}^{m}\times S^{m-1}$;}
\item[(b)]{\underline{or} intersects itself infinitely often (i.e. the intersection set has infinitely many connected components) and has the topological type of $\mathbb{R}\times S^{m-1}\times S^{m-1}$.}	
\end{enumerate}	
Fouthermore, in both cases, the hypersurfaces intersect the cone $C_m$ outside any compact set and it is arbitrarily close to $C_m$.
\end{theorem}	

Both cases actually occur, and we will be interested in (a).
Specifically, for $m=2$, one can fix a generating curve $\sigma:[0,\infty)\to \mathbb{R}^{2}$ of type (a) and let us assume, without loss of generality, that the parametrization is by arclength and $\sigma(0)=(1,0)$. It follows from Alencar's dicussion that the rescalings of the pre-image $E=\pi_2^{-1}(\textrm{spt}(\sigma))$ given by
$\lambda^{-1}E$
 converge  the the Clifford cone in the sense of varifolds as we let $\lambda\to\infty$. The convergence happens locally in the sense of smooth graphs away from the singuularities. Lastly, let us explicitly remark that for any $r>1$ the intersection $E\cap B_{r}$ is diffeomorphic to $D^{2}\times S^1$, namely a handle-body.

In the case of the four-sphere with the round metric $(S^4,\gamma_0)$, a similar ODE analysis was performed by Hsiang (see \cite{Hsi83a, Hsi83b}) in order to produce desingularizations of the Clifford football. However, his results are global and (as inticipated in the introduction) ensure the existence a sequence of embedded minimal hyperspheres that converge to $M$ uniformly away from the poles. From a local perspective, namely on small geodesic balls in $(S^4,\gamma_0)$ centered at the singularities of the Clifford football, the Hsiang regularizations can be seen as small perturbations of the Alencar regularizations. As a result, their properties mirror those of type (a) solutions of Theorem \ref{thm:Ale}.
The results we shall need in the sequel of this article are collected in the following statement.

\begin{theorem}(Theorem 1 in \cite{Hsi83a})\label{thm:Hsi}
For each positive odd integer $2i+1$, there exists a $G$-invariant, minimal embedding $E_i$ of $S^{3}$ into $S^4$ whose image curve $E_i/G$ is central symmetric with respect to the center point $(\pi/2,\pi/4)$ in $S^4/G$ and intersects with the bisector at eaxctly $2i+1$ points. Furthermore, the sequence $E_i$ converges to the Clifford football $M$ uniformly on any given compact set disjoint from the poles of $S^4$  in the sense of smooth graphs. Finally, the products $\sin(\theta)|A_i|$ (for $A_i$ the second fundamental form of $E_i$ in $(S^4,\gamma_0))$ are uniformly bounded independently of $i$.
\end{theorem}	

The first and second statement are proven in \cite{Hsi83a} (in particular the latter is remarked in Section 5, (3)). The third statement also follows from the ODE analysis by Hsiang, but can also be deduced from the results by Alencar arguing by contradiction by means of a blow-up argument. 

\section{Deformation theory}\label{sec:deform}

The scope of this section is to prove Theorem \ref{pro:deform}, which ensures the existence of deformed Clifford footballs in all those conformal classes that have nearly-round representatives.

\subsection{Bipolar conformal factors}\label{subs:bipo}

We shall start here by describing the construction of the conformal factors that enter into the definition of the map $\Xi$ and that, as a matter of fact, play a key role in our approach. 
We let from now onwards $\rho:M\to\mathbb{R}$ denote a fixed radius function for the Clifford football $M\hookrightarrow (S^{4},\gamma_0)$. For the sake of definiteness the reader might simply consider a smoothing of the function $\min\left\{d_{\gamma_0}(p_1,\cdot), d_{\gamma_0}(p_2,\cdot),1\right\}$ and in our case $\rho=\sin(\theta)$ is a natural choice. Furthemore, for $\varepsilon>0$ small enough we set 
\[
U=\left\{p\in S^4 \ : \ d_{\gamma_0}(p,M)<2\varepsilon\rho(p)\right\}
\]
where of course $d_{\gamma_0}(p,M)=\inf_{q\in M}d_{\gamma_0}(p,q)$.
Depending on such $\varepsilon$ we let $\chi=\chi_{\varepsilon}:\mathbb{R}\to\mathbb{R}$ be a smooth non-increasing function that equals $1$ for $t\leq\varepsilon$ and $0$ for $t\geq2\varepsilon$. We then define $\hat{\chi}:S^{4}\setminus\left\{p_1,p_2\right\}\to\mathbb{R}$ by $\hat{\chi}=\chi\circ \left(d_{\gamma_0}(M,\cdot)/\rho(\cdot)\right)$.
We observe that one can conveniently describe the points in the \textsl{conical neighbourhood} $U$ by means of a couple of coordinates $(z,s)$ where $z=z(\phi,\psi,\theta)$ parametrizes the Clifford football and $s$ is the signed distance from it.

If $B\in\mathbb{R}$ and $u\in \mathcal{W}^{k,2}_{\boldsymbol{\beta}}(M)$
and $\boldsymbol{\beta}=(\beta,\beta)$ with $\beta>1$, we consider the conformal factor
\[
Q(u)(z,s)=\left(1+\frac{Bsu(z)}{\sin^{2}(\theta)}\right)^{2}
\]
which can be seen to extend, by means of the cut-off function $\hat{\chi}$, to a $\mathcal{C}^{q,\alpha'}$ function on $S^4$ provided $\beta>1+\alpha'+q, \ k>3/2+q$  and $\|u\|_{\mathcal{W}^{k,2}_{\boldsymbol{\beta}}}$ is small enough. Here we have used the embedding Theorem \ref{thm:Sob}. In particular, given $\delta>0$ if we require 
\[
\|u\|_{\mathcal{W}^{k,2}_{\boldsymbol{\beta}}}<\frac{\delta}{2\varepsilon|B|}
\]
we can ensure that $Q(u)$ only attains values in the range $[(1-\delta)^{2},(1+\delta)^{2}]$ so that we can use it as a conformal factor to perturb a given Riemannian metric on $S^{4}$.
Any such function $Q(u)$ satisfies two important properties:
\begin{enumerate}
\item{$Q(u)=1$ identically on the Clifford football $M$;}
\item{$\gamma_0(\nabla Q(u),\nu)=\frac{2Bu(z)}{\sin^{2}(\theta)}$ on the Clifford football $M$.}	
\end{enumerate}	

\subsection{Idea of the approach}\label{subs:idea}

Given $q\geq 3, \ \boldsymbol{\beta}=(\beta,\beta)$ with $\beta>1+\alpha'+q$ and $B\in\mathbb{R}$ to be chosen in a suitable way as we are about to describe, we let $\hat{\Gamma}(q+1,\alpha)=\Gamma(q+1,\alpha)\cap B_{\delta}(\gamma_0)$ and $\hat{\mathcal{W}}=\mathcal{W}^{q+2,2}_{\boldsymbol{\beta}}(M)\cap B_{\delta/(2\varepsilon |B|)}(0)$. Correspondingly, we consider the map $\mathcal{M}:\hat{\Gamma}\times\hat{\mathcal{W}}\to \mathcal{W}^{q,2}_{\boldsymbol{\beta}-2}(M)$ given by
\[
\mathcal{M}(\gamma,u)=(Q(u))^{-1/2}\left(H_{\gamma}(u)+\frac{1}{2}\nabla^{\gamma}_{\nu}\log Q(u)\right)
\] 
where $H_{\gamma}(u)$ denotes the mean curvature, with respect to the Riemannian metric $\gamma$ on $S^{4}$ of the normal graph over the Clifford football defined by the function $u$.
In geometric terms, this functional gives the mean curvature of such graph with respect to the conformally deformed metric $Q(u)\gamma$.
It is readily checked that the map $\mathcal{M}:\hat{\Gamma}\times\hat{\mathcal{W}}\to \mathcal{W}^{q,2}_{\boldsymbol{\beta}-2}(M)$ is $\mathcal{C}^{1}$ in the sense of Calculus in Banach spaces, and its partial derivative with respect to the \textsl{second} slot evaluated at the point $\gamma=\gamma_0, u=0$ is given by
\[
\mathcal{M}_{u}(\gamma_0,0)[v]=J_M v+B\sin^{-2}(\theta)v
\] 
where we have made use of the property (1) stated in the previous subsection. Using the explicit expression that has been derived for the Jacobi operator of the Clifford football in Subsection \ref{subs:Jacobi} and property (2) we get
\[
\mathcal{M}_{u}(\gamma_0,0)[v]=\frac{2}{\sin^{2}\theta}\frac{\partial^{2}\overline{v}}{\partial\phi^{2}}+\frac{2}{\sin^{2}\theta}\frac{\partial^{2}\overline{v}}{\partial\psi^{2}}+\frac{\partial^{2}\overline{v}}{\partial\theta^{2}}+(2\cot\theta)\frac{\partial \overline{v}}{\partial\theta}+\left(3+\frac{(2+B)}{\sin^{2}\theta}\right)\overline{v}.\]

\subsection{Analysis of the singular Jacobi operator}

We are ready to prove Theorem \ref{pro:deform}.

\begin{proof}
For $q, \boldsymbol{\beta}$ as above, we claim that one can determine the constant $B$ in a way that the linearized operator $\mathcal{M}_{u}(\gamma_0,0):\mathcal{W}^{q+2,2}_{\boldsymbol{\beta}}(M)\to\mathcal{W}^{q,2}_{\boldsymbol{\beta}-2}(M)$ is a Banach space isomorphism. More specifically, we claim that has to be the case once we set $B=-2(1+b^{2})$ for 
\[
b>b_{\ast}:=\max\left\{\sqrt{\frac{3}{2}};\sqrt{\frac{1}{2}\left(\beta+\frac{1}{2}\right)^{2}-\frac{1}{8}}\right\}
\]
and in fact we shall set $b=2b_{\ast}$ for the sake of definiteness. First of all, it is clear that in this range the operator $\mathcal{M}_{u}(\gamma_0,0)$ has to be injective. Indeed, let $\mathcal{M}_{u}(\gamma_0,0)[v]=0$: since $\beta>0$ (and $u\in\mathcal{C}^{3}$) we know that $v$ decays on approach to the singular points of the Clifford football, so (if it is not identically zero) then possibly changing its sign we also know that $v$ attains a global maximum on $\dot{M}$ and hence a standard application of the maximum principle on a relatively compact subdomain of $\dot{M}$ (where $\mathcal{M}_{u}(\gamma_{0},0)$ is uniformly elliptic and has uniformly bounded coefficients) implies that the operator is injective. Here we have used the assumption that $b>\sqrt{\frac{3}{2}}$. Let us now discuss the surjectivity.  To this end, for the sake of clarity let us denote $T=\mathcal{M}_{u}(\gamma_0,0):\mathcal{W}^{q+2,2}_{\boldsymbol{\beta}}(M)\to\mathcal{W}^{q,2}_{\boldsymbol{\beta}-2}(M)$ and its dual by $T^{\ast}:\mathcal{W}^{-q,2}_{-\boldsymbol{\beta}-1}(M)\to\mathcal{W}^{-q-2,2}_{-\boldsymbol{\beta}-3}(M)$, where we have used the well-known identification $\left(\mathcal{W}^{k,p}_{\boldsymbol{\beta}}\right)^{\ast}\simeq \mathcal{W}^{-k,p'}_{-\boldsymbol{\beta}-m}$ for $m$ the dimension of the underlying conifold (cmp. for instance Section 7 and Section 9 in \cite{Pac13}).
Observe that $T$ is formally self-adjoint so we have at once that for $\beta>0$ the operator $T^{\ast}$ has to be surjective. Now, it follows from Lockhart-McOwen theory (see \cite{LMc85, Loc87}) that the operator $T:\mathcal{W}^{q+2,2}_{\boldsymbol{\beta}}(M)\to\mathcal{W}^{q,2}_{\boldsymbol{\beta}-2}(M)$ is Fredholm provided $\beta$ is \textsl{not} an indicial root, which by standard separation of variables (as in \cite{CHS84}) reduces the issue to checking that $\beta$ is \textsl{not} a root of the polynomial
\[
 P_{p,q}(t)=t^{2}+t-2(p^{2}+q^{2}+b^{2}) \ \ \textrm{for any} \ (p,q)\in\mathbb{Z}_{\geq 0}\times\mathbb{Z}_{\geq 0}.
 \]
 It is then straightforward to check that our assumption on $b$ implies that both $T$ and $T_{\ast}$ (which is nothing but the \textsl{same} operator, acting between different Banach spaces) are Fredholm. Furthermore, it is also a standard result (the reader may wish to consult Theorem 7.9 in \cite{Pac13} for a precise statement) that the \textsl{difference} of the Fredholm indices of $T$ and $T^{\ast}$ (which we shall denote by $\textbf{FI}(T)$ and $\textbf{FI}(T^{\ast})$ respectively) is given by the weight-crossing formula
	\[
	\textbf{FI}(T)-\textbf{FI}(T^{\ast})=\textbf{FI}_{\boldsymbol{\beta}}(T)-\textbf{FI}_{-\boldsymbol{\beta}-1}(T)=\sum_{\underline{\zeta}\in\mathcal{D}_{T_{\infty}}, -\boldsymbol{\beta}-1\leq\underline{\zeta}\leq\boldsymbol{\beta}}\textbf{m}_{T_{\infty}}(\underline{\zeta})
	\] 
where the right-hand side accounts for the dimensions of the eigenspaces associated to the indicial roots between the weights $-\boldsymbol{\beta}-1$ and $\boldsymbol{\beta}$. In our case, the requirement that $b>\sqrt{\frac{1}{2}\left(\beta+\frac{1}{2}\right)^{2}-\frac{1}{8}}$ is equivalent to $\beta<-\frac{1}{2}+\sqrt{\frac{1}{4}+2b^{2}}$ as well as $-\beta-1>-\frac{1}{2}-\sqrt{\frac{1}{4}+2b^{2}}$ and hence the formula in question implies that $\textbf{FI}(T)=\textbf{FI}(T^{\ast})$. On the other hand, we already know that $\textbf{FI}(T)\leq 0$ (by injectivity) and $\textbf{FI}(T^{\ast})\geq 0$ (by surjectivity) so we conclude $\textbf{FI}(T)=0$, which means that $\mathcal{M}_{u}(\gamma_{0},0):\mathcal{W}^{q+2,2}_{\boldsymbol{\beta}}(M)\to\mathcal{W}^{q,2}_{\boldsymbol{\beta}-2}(M)$ is an isomorphism. As a result, we are in position to apply the Implicit Function Theorem in order to produce an arc of solutions of the nonlinear equation $\mathcal{M}(\gamma,u)=0$ in a neighbourhood of $(\gamma_0,0)$.  Hence, we are granted the existence of neighbourhoods $\mathcal{U}$ of $\gamma_0$ and $\mathcal{W}$ of the \textsl{zero} function in $\mathcal{W}^{q+2,2}_{\boldsymbol{\beta}}(M)$, and of a $\mathcal{C}^{1}$ map $\Omega:\mathcal{U}\to\mathcal{W}$ such that $\mathcal{M}(\gamma,\Omega(\gamma))=0$ identically in a neighbourhood of $\gamma_0$ (and this is, locally, a parametrization of \textsl{all} the solutions of such equation). Letting $\mathcal{V}$ be the image of $\mathcal{U}$ via the map $\Xi(\cdot)$ defined by $\Xi(\gamma)=Q(\Omega(\gamma))\gamma$ the proof is complete.
	
\end{proof}	

\section{Desingularization theory}\label{sec:desing}

In this Section we will prove Theorem \ref{thm:main} and Theorem \ref{thm:main2} by desingularizing the perturbed Clifford football that have been produced above (Theorem \ref{pro:deform}).  

\subsection{Coarse interpolation}\label{subs:coarse}

Let $\gamma\in\mathcal{U}$ a \textsl{fixed} Riemannian metric: we already know that the normal graph over the Clifford football determined by the function $\Omega(\gamma)$ is minimal with respect to the conformally deformed metric $Q(\Omega(\gamma))\gamma$: its closure, which we shall denote by $M_{\Omega(\gamma)}$ is a minimal submanifold with two conical singularities in the sense of Definition \ref{def:singman}.  We also observe that, by the way our construction has been performed, the tangent cones to $M_{\Omega(\gamma)}$ at the poles are Clifford cones, in fact the same as for the Clifford football $M$.

\

Let us denote by $p_1$ (respectively $p_2$) the north (resp. south) pole of the sphere $S^4$. For each of them (and, for the sake of clarity, let us agree to work with the north pole) let $\left\{w\right\}$ be a system of geodesic normal coordinates on $B_{\eta_2}(p_1)$ for some small $\eta_2$ to be determined later. Without loss of generality (possibly by acting via a Euclidean isometry) we can assume that the Hsiang regularizations converge (in varifold sense) to the tangent cone of $M$ at the poles.
Let $\xi:\mathbb{R}^{4}\to\mathbb{R}$ be a smooth, radial cut-off function that equals one on the ball of radius $\eta_2/2$ and zero outside of the ball of radius $\eta_2$: using those local coordinates we can perform a coarse interpolation of $M_{\Omega(\gamma)}$ and $E_i$, thereby obtaining a closed four-manifold $\tilde{M}_{\eta_1,\eta_2}$ which coincides with the former outside of the balls of radius $\eta_2$ around each pole and instead coincides with the latter inside the balls of radius $\eta_2/2$. For purely notational convenience, we have introduced the \textsl{discrete} parameter $\eta_1$ as a replacement for the index $i$: $\eta_1$ corresponds to the distance from the origin to the complete sumbmanifold $E_i$ and controls, at the same time, the order of its convergence to the Clifford cone.  Of course, we can perform the interpolation because both summands of the connected sum in question are described by normal graphs over the Clifford cone $\left\{w_1^{2}+w_2^{2}=w_3^{2}+w_4^{2}\right\}\subset\mathbb{R}^{4}$ (at least in the annulus $B_{\eta_2}\setminus B_{\eta_2/2}$). The resulting four-manifold is obviously minimal (with respect to $\Xi(\gamma)=Q(\Omega(\gamma))\gamma$) away from the balls around the conical singularities, while it will not be inside. However, as we let $\eta_1\to 0$ (for fixed, small $\eta_2$) the mean curvature of $\tilde{M}_{\eta_1,\eta_2}$ inside those spherical caps will converge to zero, which is the \textsl{heuristic} idea that motivates the iterative scheme we are about to describe. 

In order to write the Schauder estimates we need in a convenient fashion, namely with constants that do not depend on the gluing parameters $\eta_1, \eta_2$ it is useful to consider on $\tilde{M}_{\eta_1,\eta_2}$ \textsl{extrinsincally weighted} functional spaces. In particular, for $\beta\in\mathbb{R}$ we shall consider on $\mathcal{C}^{k,\alpha'}(\tilde{M}_{\eta_1,\eta_2})$ the norm
\[
\left\|u\right\|_{\mathcal{C}^{k,\alpha'}_{\beta}}:=\sum_{j=0}^{k}\sup_{\tilde{M}_{\eta_1,\eta_2}}\rho^{-\beta+j}|\nabla^{j}u| + \sup_{x\neq y \in \tilde{M}_{\eta_1,\eta_2}}\frac{|\rho^{-\beta+k}(x)\nabla^{k}u(x)-\rho^{-\beta+k}(y)\nabla^{k}u(y)|}{d(x,y)^{\alpha'}}.
\]	
Of course, for fixed values of $\eta_1,\eta_2$ such norm is patently equivalent to the standard $\mathcal{C}^{k,\alpha}$-norm on $\tilde{M}_{\eta_1,\eta_2}$. However, such equivalent is not, by any means, uniform in the gluing parameters $\eta_1,\eta_2$ and thus this setup will simplify our discussion.

\subsection{Setting up the problem.}
Similarly to what we had done in Subsection \ref{subs:bipo}, we shall introduce here a suitable conformal deformation of the mean curvature operator. To that aim, we perform the following constructions:
\begin{itemize}
\item{\textbf{Shrinking tubular neighborhoods:} for $\varepsilon>0$ a small parameter, we consider tubular neighborhoods $\tilde{U}_{\eta_1,\eta_2}$ of $\tilde{M}_{\eta_1,\eta_2}$ in $(S^4,\tilde{\gamma})$ whose width around some $p\in \tilde{M}_{\eta_1,\eta_2}$ is of order $2\varepsilon\rho(p)$ (here $\tilde{\gamma}$ is a metric on $S^4$ that is very close to $\gamma_0$, and a posteriori we will set $\tilde{\gamma}=\Xi(\gamma)$ as provided by Theorem \ref{pro:deform}
and $\rho$ is a radius function for the background metric we are working with, that is $\tilde{\gamma}$).}	
\item{\textbf{Conformal factors:} any such tubular neighborhood is patently diffeomorphic to the product $\tilde{M}_{\eta_1,\eta_2}\times [0,1)$ and hence we can introduce coordinates $(z,s)$ in the obvious way. Similarly to what we had done in Subsection \ref{subs:bipo}, given a function $u\in \mathcal{C}^{k,\alpha'}_{\beta}(\tilde{M}_{\eta_1,\eta_2})$ we can define on $\tilde{U}_{\eta_1,\eta_2}$ the conformal factor
	\[
	\tilde{Q}_{\eta_1,\eta_2}(u)(z,s)=\left(1+\frac{Bsu(z)}{\rho^2}\right)^2
	\]
where $B\in\mathbb{R}$ is as above, namely $B=-1-b^2$ for $b>\tilde{b}_{\ast}$ to be specified below. It is readily seen that such factor can be extended to a $\mathcal{C}^{q,\alpha'}$ function on the whole ambient manifold by means of a cut-off function provided $\beta>1+\alpha'+q, \ k\geq q$ and, furthermore, if the norm of $u$ is small enough, namely $\left\|u\right\|_{\mathcal{C}^{k,\alpha'}_{\beta}}<\delta/(2\varepsilon|B|)$ we are ensured that $\tilde{Q}_{\eta_1,\eta_2}(u)$ only attains values in the range $[(1-\delta)^2,(1+\delta)^2]$ and thus we may legitimately use it as a conformal factor. Once again, there are two key properties of $\tilde{Q}$ we shall exploit:
\begin{enumerate}
\item[$\tilde{(1)}$]{$\tilde{Q}_{\eta_1,\eta_2}(u)=1$ identically on $\tilde{M}_{\eta_1,\eta_2}$;}
\item[$\tilde{(2)}$]{$\tilde{\gamma}(\nabla \tilde{Q}_{\eta_1,\eta_2},\nu)=\frac{2Bu(z)}{\rho^2}$ on the closed surface $\tilde{M}_{\eta_1,\eta_2}$ (we are going to set $\tilde{\gamma}=\Xi(\gamma)$ so that such surface is minimal in metric $\Xi(\gamma)$).}	
\end{enumerate}	
}
\item{\textbf{Conformal mean curvature map:} given these constructions, and keeping in mind that we shall later take $\tilde{\gamma}=\Xi(\gamma)$ we will then consider the map
\[
\tilde{\mathcal{M}}(\tilde{\gamma},u)=\left(\tilde{Q}_{\eta_1,\eta_2}(u)\right)^{-1/2}\left(H_{\tilde{\gamma}}(u)+\frac{1}{2}\nabla^{\tilde{\gamma}}_{\nu}\log Q_{\eta_1,\eta_2}(u)\right)
\]	
where $H_{\tilde{\gamma}}(u)$ is the mean curvature (in $(S^4,\tilde{\gamma})$) of the normal graph over $\tilde{M}_{\eta_1,\eta_2}$ with defining function $u$ and $\nu$ is the unit normal to such graph (with respect to the metric $\tilde{\gamma}$). This map describes the mean curvature of such submanifold with respect to the conformally perturbed metric $Q(u)\tilde{\gamma}$. We shall then consider $\tilde{\mathcal{M}}:\mathcal{V}\times \tilde{\mathcal{C}}\to \mathcal{C}^{q-2,\alpha'}_{\beta-2}$ (where $\mathcal{V}$ is provided by Theorem \ref{pro:deform}, while $\tilde{\mathcal{C}}=\mathcal{C}^{q,\alpha'}_{\beta}(\tilde{M}_{\eta_1,\eta_2})\cap B_{\delta/(2\varepsilon |B|)}(0)$): this map is $\mathcal{C}^1$ and its partial derivative, with respect to the second argument, evaluated at the point $(\tilde{\gamma},0)$ is given by
\[
\tilde{\mathcal{M}}_{u}(\tilde{\gamma},0)[v]=J_{\tilde{M}_{\eta_1,\eta_2}}v+B\rho^{-2}v.
\]
For simplicity of notation, let us set from now onwards $L[v]:=\tilde{\mathcal{M}}_{u}(\tilde{\gamma},0)[v]$ where the dependence of the operator on the metric $\tilde{\gamma}$ and on the gluing parameters $\eta_1,\eta_2$ is left implicit.
}	
\end{itemize}	



\subsection{Solvability of the linear problem}

\begin{lemma}\label{lem:Schauder} 
Given any $q\geq 3, \alpha'\in (0,1/2), \beta>2$ and $\tilde{\gamma}\in\mathcal{V}$ (as per Theorem \ref{pro:deform}), there exist positive constants $\tilde{b}_{\ast}$ and $C$ (both independent of $\eta_1,\eta_2$ small enough) such that if we let $b>\tilde{b}_\ast$ then the operator	$L:\mathcal{C}^{q,\alpha'}_{\beta}\to\mathcal{C}^{q-2,\alpha'}_{\beta-2}$ is a uniformly coercive operator, namely
\[
\left\|v\right\|_{\mathcal{C}^{q,\alpha'}_{\beta}}\leq C \left\|Lv\right\|_{\mathcal{C}^{q-2,\alpha'}_{\beta-2}}.
\]
\end{lemma}	

\begin{proof}
 We know that the operator $L:\mathcal{C}^{q,\alpha'}_{\beta}\to\mathcal{C}^{q-2,\alpha'}_{\beta-2}$ is uniformly elliptic (since its principal symbol is that of the Laplace operator on $\tilde{M}_{\eta_1,\eta_2}$) and in our setting the weighted Schauder estimates take the form
\[
\left\|v\right\|_{\mathcal{C}^{q,\alpha'}_{\beta}}\leq C\left(\left\|v\right\|_{\mathcal{C}^{0}_{\beta}}+ \left\|Lv\right\|_{\mathcal{C}^{q-2,\alpha'}_{\beta-2}}\right)
\]	
for a constant $C$ that does not depend on the gluing parameters. The key claim is that, for a suitable choice of $b$ necessarily
\[
\left\|v\right\|_{\mathcal{C}^{0}_{\beta}}\leq \left\|Lv\right\|_{\mathcal{C}^{0}_{\beta-2}}
\]
which would immediately imply the conclusion. Let us pick $b>2\tilde{b}_{\ast}$ where $\tilde{b}_{\ast}$ is chosen so that 
\[
|\tilde{A}_{\eta_1,\eta_2}|^2+Ric(\nu,\nu)\leq \frac{\tilde{b}_{\ast}^2}{\rho^2} \ \ \textrm{on} \ \tilde{M}_{\eta_1,\eta_2}
\]
which we can do (uniformly in $\eta_1$, for $\eta_2$ sufficiently small) because of the last assertion in Theorem \ref{thm:Hsi}. Then, our claim would be implied by showing that in fact

\[
 \sup_{\tilde{M}_{\eta_1,\eta_2}}\left| \rho^{-\beta+2}\left(\Delta v-3\tilde{b}_{\ast}^2\rho^{-2}v\right)\right|\geq \sup_{\tilde{M}_{\eta_1,\eta_2}}\left| \rho^{-\beta} v\right|.
\]
To that aim, let us consider a point $\tilde{x}_{\ast}\in\tilde{M}_{\eta_1,\eta_2}$ where the value attained by the quantity $\left| \rho^{-\beta}v\right|$ is maximum. Without loss of generality we can assume that $v(\tilde{x}_\ast)>0$ for otherwise the argument is symmetric. Now, we know that $\Delta(\rho^{-\beta}v)(\tilde{x}_{\ast})\leq 0$ as well as $\nabla(\rho^{-\beta}v)(\tilde{x}_{\ast})=0$ (which allows to express $\nabla v(\tilde{x}_{\ast})$ in terms of $v(\tilde{x}_{\ast})$) and thus we have the chain of inequalities 
\[
\rho^{-\beta+2}(\Delta v(\tilde{x}_{\ast})-3\tilde{b}_{\ast}^2\rho^{-2}v(\tilde{x}_{\ast}))=\rho^2\left[\rho^{-\beta}\Delta v(\tilde{x}_{\ast})-3\tilde{b}_{\ast}^2\rho^{-\beta-2}v(\tilde{x}_{\ast})\right]
\]
\[
\leq \rho^{2}[\Delta(\rho^{-\beta}v(\tilde{x}_{\ast}))-\tilde{b}_{\ast}^2\rho^{-\beta-2}v(\tilde{x}_{\ast})]\leq -\rho^{-\beta}v(\tilde{x}_{\ast})
\]
where in the second to last step we might have to take $\tilde{b}_{\ast}$ bigger than we had done, this depending on the parameter $\beta$ only. This shows that
\[
\sup_{\tilde{M}_{\eta_1,\eta_2}}\left| \rho^{-\beta} v\right|\leq \left|\rho^{-\beta+2}\left(\Delta v(\tilde{x}_{\ast})-3\tilde{b}_{\ast}^2\rho^{-2}v(\tilde{x}_{\ast})\right)\right|
\]
and hence, to greater extent
\[
\sup_{\tilde{M}_{\eta_1,\eta_2}}\left| \rho^{-\beta} v\right|\leq \sup_{\tilde{M}_{\eta_1,\eta_2}}\left|\rho^{-\beta+2}\left(\Delta v-3\tilde{b}_{\ast}^2\rho^{-2}v\right)\right|.
\]
This proves our key claim and thus the statement of the lemma.
\end{proof}

\subsection{A Picard iteration scheme for the nonlinear problem}\label{subs:Pic}

We have just seen that the constant $B$ can be chosen so that the linearization of the operator $\tilde{\mathcal{M}}(\tilde{\gamma},0)$ for $\tilde{M}_{\eta_1,\eta_2}$ is a linear isomorphism and its injectivity constant does not deteriorate as we let $\eta_1\to 0$ (for fixed $\eta_2$ small). Thus, there is a uniformly continuous solution operator $L^{-1}: X_1\to X_2$ where $X_1:=\mathcal{C}^{q-2,\alpha'}_{\beta-2}$ and $X_2:=\mathcal{C}^{q,\alpha'}_{\beta}$ (coherently, the symbol $\left\|\cdot\right\|_1$ (resp. $\left\|\cdot\right\|_2$) shall stand for the Banach norm in $\mathcal{C}^{q-2,\alpha'}_{\beta-2}$ (resp. $\mathcal{C}^{q,\alpha'}_{\beta}$)).  We shall then approach the solvability of the nonlinear problem
\[
\tilde{\mathcal{M}}(\tilde{\gamma},v)=0
\] 
by means of a Picard iteration scheme. More specifically, we shall write the equation in question as
\[
\tilde{\mathcal{M}}(\tilde{\gamma},0)+\tilde{\mathcal{M}}_u(\tilde{\gamma},0)[v]+\tilde{\mathcal{N}}(\tilde{\gamma},0)(v)=0
\] 
where $\tilde{\mathcal{N}}(\tilde{\gamma},0)$ collects all the terms of $\tilde{\mathcal{M}}(\tilde{\gamma},v)$ that are not linear in $v$, hence in fact at least quadratic. Let us recall that we have conveniently set $Lv=\tilde{\mathcal{M}}_u(\tilde{\gamma},0)[v]$ and similarly, we shall define here $Z(v):=\tilde{\mathcal{N}}(\tilde{\gamma},0)(v)$ The iteration we shall setup is defined by letting:
\[
\begin{cases}
u_0=0 \\
f_0=-\tilde{\mathcal{M}}(\tilde{\gamma},0)
\end{cases}
\]
and hence, recursively,
\[
\begin{cases}
u_{i+1}=L^{-1}(f_i) \\
f_{i+1}=-\tilde{\mathcal{M}}(\tilde{\gamma},0)-Z(u_{i+1}).
\end{cases}
\]
In doing this, we need to make sure that the remainder terms (namely those at least quadratic, represented by $Z$) get smaller and smaller along the iteration, so that the method converges. To that aim, the following statement does suffice.
\begin{proposition} \label{pro:quadratic} Given any $\lambda>0$, there exists $r_0>0$ sufficiently small so that 
	if $\left\|f_1\right\|_1<r_0$ and $\left\|f_2\right\|_1<r_0$ and we let $u_1=L^{-1} f_1$, $u_2=L^{-1} f_2$ then we have
	\[ \left\|Z(u_1)-Z(u_2)\right\|_1\leq \lambda \left\|u_1-u_2\right\|_2.
	\] 
\end{proposition}

Before discussing why this has to be the case in our problem, let us show how such control on the $Z$ term implies convergence of the scheme in the space $X_1$. 

  \begin{proposition} \label{pro:picard} Given $f\in X_1$ sufficiently small, there is a small
  	$u\in X_2$ satisfying 
  	\[ Lu+Z(u)=f.
  	\]
  \end{proposition}
  
  \begin{proof} Assume that $\|f\|_{1}<\delta_0$ (with $\delta_{0}$ a small constant to be fixed later in the proof), let $u_0=0$ and $f_0=f$, and we
  	inductively construct sequences $f_i$ and $u_i$ for $i\geq 1$ such that
  	\[ Lu_i=f_{i-1}\ \mbox{where}\ f_i=-Z(u_{i})+f.
  	\] 
  	For $i\geq 1$ we have
  	\[ L(u_{i+1}-u_i)=f_{i}-f_{i-1}=Z(u_{i-1})-Z(u_i),
  	\]
  	and so by Proposition \ref{pro:quadratic} we have
  	\[ \|f_{i+1}-f_i\|_1=\|Z(u_{i+1})-Z(u_{i})\|_1\leq \lambda\|u_{i+1}-u_{i}\|_2\leq C\lambda\|f_i-f_{i-1}\|_1
  	\]
  	where $\lambda$ can be chosen as small as we wish and $C$ is the continuity constant of the solution operator $L^{-1}$ (this can be chosen uniformly thanks to Lemma \ref{lem:Schauder}). Let then $r_0$ be small enough so that
  	$C\lambda<1/2$ in Proposition \ref{pro:quadratic}. We may then iterate this scheme provided that
  	$\|f_i\|_1\leq r_0$ for $i=1,\ldots, k$ and in that case we obtain 
  	\[ \|f_{k+1}-f_k\|_1\leq 2^{-k-1}\|f_1-f_0\|=2^{-k-1}\|f\|_1< 2^{-k-1}\delta_0
  	\]
  	From the triangle inequality we then have for any $k$
  	\[ \|f_{k+1}-f\|_1\leq \sum_{i=1}^{k+1}2^{-i}\delta_0<2\delta_0,
  	\]
  	so if we choose $\delta_0=r_0/4$ we have
  	\[ \|f_{k+1}\|_1\leq \|f_{k+1}-f\|_1+\|f\|_1<3\delta_0<r_0
  	\]
  	for each $k$. We can then iterate indefinitely and the sequence $\{f_i\}$ is Cauchy
  	as is $\{u_i\}$ since $L^{-1}$ is a bounded operator. As a consequence, the sequence $\{u_i\}$ converges in $X_{2}$ to a limit $u$ which satisfies
  	the equation $Lu+Z(u)=f$. This completes the proof.
  \end{proof}
  
  At this stage, we shall outline the proof of Proposition \ref{pro:quadratic}.
  
  \begin{proof}
Thanks to the computation presented (for instance) in section 7.1 of \cite{CM11}, where the local expression of the mean curvature map for an hypersurface in a Riemannian manifold is derived, we know that the value of $Z(v)$ will be bounded from above by a finite sum of terms of the form $\prod_{i\in I} \nabla^{\epsilon_i}v$ where $\epsilon_i$ represents differentiation of order $|\epsilon_i|$ under the constraint that $|\epsilon_i|\leq 2$ for each value of the index $i\in I$.  Thus in order to estimate
\[
\sup_{\tilde{M}_{\eta_1,\eta_2}}\rho^{-\beta+2}|Z(u_1)-Z(u_2)|
\] 
it is in fact enough to show that
\[
\sup_{\tilde{M}_{\eta_1,\eta_2}} \rho^{-\beta+2}|\nabla^{\epsilon'}u_1\nabla^{\epsilon''}u_1-\nabla^{\epsilon'}u_2\nabla^{\epsilon''}u_2|\leq C r_0 \left\|u_1-u_2 \right\|_{X_2}.
\]
This is indeed the case, for the triangle inequality gives
\[
 \rho^{-\beta+2}(x)|\nabla^{\epsilon'}u_1(x)\nabla^{\epsilon''}u_1(x)-\nabla^{\epsilon'}u_2(x)\nabla^{\epsilon''}u_2(x)| 
 \]
 \[
 \leq  \rho^{-\beta+2}(x)|\nabla^{\epsilon'}u_1(x)\nabla^{\epsilon''}u_1(x)-\nabla^{\epsilon'}u_1(x)\nabla^{\epsilon''}u_2(x)|+\rho^{-\beta+2}(x)|\nabla^{\epsilon'}u_1(x)\nabla^{\epsilon''}u_2(x)-\nabla^{\epsilon'}u_2(x)\nabla^{\epsilon''}u_2(x)| \]
 \[
\leq  \rho^{-|\epsilon''|+2}(x)|\nabla^{\epsilon'}u_1(x)|\rho^{-\beta+|\epsilon''|}(x)|\nabla^{\epsilon''}u_1(x)-\nabla^{\epsilon''}u_2(x)|
\]
\[
+\rho^{-|\epsilon'|+2}(x)|\nabla^{\epsilon''}u_2(x)|\rho^{-\beta+|\epsilon'|}(x)|\nabla^{\epsilon'}u_1(x)-\nabla^{\epsilon'}u_2(x)|
\]
and hence, taking the supremum over $x\in\tilde{M}_{\eta_1,\eta_2}$ the previous inequality yields
\[
\sup_{\tilde{M}_{\eta_1,\eta_2}} \rho^{-\beta+2}|\nabla^{\epsilon'}u_1\nabla^{\epsilon''}u_1-\nabla^{\epsilon'}u_2\nabla^{\epsilon''}u_2|
\]
\[
\leq \left\{\sup_{x\in \tilde{M}_{\eta_1,\eta_2}}\rho^{-|\epsilon''|+2}(x)|\nabla^{\epsilon'}u_1(x)|+\rho^{-|\epsilon'|+2}(x)|\nabla^{\epsilon''}u_2(x)| \right\} \left\|u_1-u_2 \right\|_{X_2}.
\]
But then since $\beta> 1+\alpha'+q> 2$ (which we had assumed since the very definition of the spaces $X_1$ and $X_2$) patently
\[
\sup_{x\in \tilde{M}_{\eta_1,\eta_2}}\rho^{-|\epsilon''|+2}(x)|\nabla^{\epsilon'}u_1(x)|\leq \sup_{x\in \tilde{M}_{\eta_1,\eta_2}}\rho^{-\beta+|\epsilon'|}(x)|\nabla^{\epsilon'}u_1(x)|\leq C r_0
\]
as well as
\[
\sup_{x\in \tilde{M}_{\eta_1,\eta_2}}\rho^{-|\epsilon'|+2}(x)|\nabla^{\epsilon''}u_2(x)|\leq \sup_{x\in \tilde{M}_{\eta_1,\eta_2}}\rho^{-\beta+|\epsilon''|}(x)|\nabla^{\epsilon''}u_2(x)|\leq C r_0
\]
so that, in the end
\[
\sup_{\tilde{M}_{\eta_1,\eta_2}} \rho^{-\beta+2}|\nabla^{\epsilon'}u_1\nabla^{\epsilon''}u_1-\nabla^{\epsilon'}u_2\nabla^{\epsilon''}u_2|\leq 2 C r_0 \left\|u_1-u_2 \right\|_{X_2}.
\]
The estimate for the covariant derivatives and for the term
\[
\sup_{x\neq y \in \tilde{M}_{\eta_1,\eta_2}}\frac{|\rho^{-\beta+2+\alpha'}(x)\nabla^{q-2}(Z(u_1)-Z(u_2)(x))-\rho^{-\beta+2+\alpha'}(y)\nabla^{q-2}(Z(u_1)-Z(u_2))(y)|}{d(x,y)^{\alpha'}}
\]
follows along similar lines, the latter just by exploiting the inequality
\[ [f_1f_2]_{\alpha'}\leq \left(\sup_{\tilde{M}_{\eta_1,\eta_2}}|f_1|\right)[f_2]_{\alpha'}+\left(\sup_{\tilde{M}_{\eta_1,\eta_2}}|f_2|\right)[f_1]_{\alpha'}.
\]
Adding up the two terms, we end up proving an inequality of the form 
  	\[ \left\|Z(u_1)-Z(u_2)\right\|_1\leq C r_0\left\|u_1-u_2\right\|_2.
  	\] 
 where $C$ is a constant that only depends on the background metric, and can be chosen uniformly in $\mathcal{V}=\Xi(\mathcal{U})$ ($\mathcal{U}$ being a small neighborhood of the round metric $\gamma_0$ on $S^4$). Therefore, by simply letting (for such a constant) $r_0=\lambda/C$ the proof is complete. 	
  	\end{proof}

\subsection{Proof of Theorem \ref{thm:main}}

In this subsection, we collect and make use of all the intermediate results that have been obtained in the article in order to give a direct proof of Theorem \ref{thm:main}. 

\begin{proof}
	
For a given couple of antipodal points of the four-sphere (which, without loss of generality, we shall assume to be the north and the south poles) let $\mathcal{O}=\cup_O\mathcal{O}(O)$ be the union of the $\mathcal{C}^{q+1,\alpha}$-Riemannian metrics on $S^4$ whose Weyl tensor vanishes on some open set $O$ containing those points, as $O$ varies. We claim that one can reduce to the smaller class $\mathcal{O}'$ of those metrics that are round on some open $O$ containing the poles, as $O$ varies. Indeed, if $\gamma\in\mathcal{O}(O)$ we know by Theorem 1.165 in \cite{Bes87} that $(S^4,\gamma)$ is conformally flat in $O$ or equivalently, by means of the stereographic projection, it is pointwise conformally equivalent to $(S^4,\gamma_0)$ on that neighborhood. That is to say that there exists a conformal factor $f=f(\gamma)\in\mathcal{C}^{q+1,\alpha}$ such that the metric $f^2\gamma$ coincides with $\gamma_0$ around the poles. In particular, if $\gamma$ is assumed to be $\mathcal{C}^{q+1,\alpha}$-close to $\gamma_0$ we will deduce that $f$ is $\mathcal{C}^{q+1,\alpha}$ close to the constant function equal to $1$. Therefore, we can apply Theorem \ref{pro:deform} \textsl{after} this preliminary step, namely after replacing each such $\gamma$ by the corresponding metric $f^2\gamma$. For notational convenience, let us simply rename $f^2\gamma$ to $\gamma$ from now onwards. Using the notation of that statement, we are then given a neighborhood $\mathcal{U}$ of Riemannian metrics about $\gamma_0$ with the property that if $[\gamma]\in\Pi(\mathcal{U})$ then $\Omega(\gamma)$ defines a singular minimal submanifold in $(S^4,\Xi(\gamma))$. This is true, as a special case, for the class of metrics $\mathcal{U}\cap\mathcal{O}'$. Now, the metric $\Xi(\gamma)$ is not exactly round around the poles anymore, but the conformal correction introduced by the map $\Xi$ satisfies the estimate $|1-Q(\Omega(\gamma))|\leq\rho^{q}$
which ensures that $\Xi(\gamma)$ agrees with the round metric at the poles up to order $q-1$. This being remarked, for any such metric $\Xi(\gamma)$ we perform the coarse interpolation described in Subsection \ref{subs:coarse}, thereby getting a closed submanifold $\tilde{M}_{\eta_1,\eta_2}$. By the well-known formula for the conformal change of the mean curvature, we know that 
\begin{equation}\label{eq:est1}
|\tilde{H}_{\eta_1,\eta_2}|\leq C \rho^{q-1} \ \textrm{in} \ \ B_{\eta_2/2} 
\end{equation}
and with little effort one can also check that
\begin{equation}\label{eq:est2}
|\tilde{H}_{\eta_1,\eta_2}|\leq C \left(h_{\eta_2}(\eta_1)+\rho^{q-1}\right) \ \textrm{in} \ \ B_{\eta_2}\setminus B_{\eta_2/2} 
\end{equation}
where $h_{\eta_2}(\cdot)$ is the modulus of continuity which encodes the rate of convergenece of the Hsiang hyperspheres to the Clifford football on approach to the singularities thereof.
Here $C$ is a constant that is independent of $\eta_1, \eta_2$ for any $\eta_2$ small enough. Now, we want to solve in $u$ the nonlinear problem
\[
\tilde{\mathcal{M}}(\tilde{\gamma},u)=0
\]
for $\tilde{\gamma}=\Xi(\gamma)$. In order to do so, we proceed by means of a Picard iteration scheme as described in Subsection \ref{subs:Pic}. Specifically, we first consider Proposition \ref{pro:quadratic} for (say) $\lambda=1/2$: given the corresponding $r_0=r_0(\lambda)$ we fix, once and for all, the parameter $\eta_2$ in a way that the norm $\left\|\tilde{H}_{\eta_1,\eta_2}\right\|_{\mathcal{C}^{q-2,\alpha'}_{q-2}}<r_0$ at least for $\eta_1$ small enough. In doing so, we first choose $\eta_2$ so that this is the case in $B_{\eta_2/2}$ and then find $\overline{\eta_1}$ so that for $\eta_1<\overline{\eta_1}$ the needed estimate is satified in $B_{\eta_2}\setminus B_{\eta_2/2}$ (Allard's regularity theorem ensures that we can gain estimates on higher and higher covariant derivatives of the mean curvature, the only constraint being the regularity of the ambient metric).
Once these choices are made, we can proceed with the iteration and Proposition \ref{pro:picard} ensures the convergence of the method. Thus for any such $\tilde{\gamma}=\Xi(\gamma)$ we have constructed a normal graph $M_{\eta_1}(\gamma)$ over $\tilde{M}_{\eta_1,\eta_2}$ that is minimal in metric $\overline{\Xi}_{\eta_1}(\gamma)$. By construction, the conformal factors we introduce in this last step are uniformly bounded in $\mathcal{C}^{q,\alpha'}$ and thus in fact the whole construction happens in a small $\mathcal{C}^{q,\alpha'}$ neighborhood of $\gamma_0$. Now, for $\alpha$ as in the statement of Theorem \ref{thm:main}, we can assume that (since the very beginning) Theorem \ref{pro:deform} had been applied for some $\alpha'>\alpha$. Hence, for any fixed $\gamma$ Arzel\'a-Ascoli ensures that we can extract a subsequence of indices $\eta_1$ (which we shall not rename) so that the corresponding conformal factors $\overline{\Xi}_{\eta_1}(\gamma)/\gamma$ converge in $\mathcal{C}^{q,\alpha}$ to a Riemannian metric $\overline{\Xi}_{\infty}(\gamma)$. The associated minimal embedded hypersurfaces $M_{\eta_1}$ (which are of course hyperspheres, since gotten by taking the connected sum of two handle-boldies near each of the two poles) converge to a minimal varifold $\boldsymbol{V}_{\infty}$ in $S^4$ which must have the north and south poles in its support. If the family $\left\{M_{\eta_1}\right\}$ had a uniforml bound on the Morse index, then by the compactness theorem of \cite{Sha15} (specifically: by Corollary A.7\footnote{The result in question is stated under the assumption of smooth convergence of the background metrics, but (as pointed out by the author) in fact $\mathcal{C}^{3,\alpha}$ convergence would suffice. The key tools for those compactness theorems are contained in the work by Schoen-Simon \cite{SS81} which deals with $\mathcal{C}^3$ hypersurfaces.}) $\boldsymbol{V}_{\infty}$ should be a smooth minimal hypersphere in $(S^4,\overline{\Xi}_{\infty}(\gamma))$ which patently contradicts the fact that the tangent varifolds $T_{p_1}\boldsymbol{V}_{\infty}$ and $T_{p_2}\boldsymbol{V}_{\infty}$ are not hyperplanes. As a result, the embedded minimal hyperspheres $\left\{M_{\eta_1}\right\}$, whose volume is bounded by (say) $2\mathscr{H}^{3}(M,\gamma_0)$ have arbitrarily large Morse indices. 
	\end{proof}	
	
\subsection{Proof of Theorem \ref{thm:main2}}

\begin{proof}
Without loss of generality, let us deal with the class of nearly-round Riemannian metrics in $\Gamma(q+1,\alpha)$ whose curvature tensor coincides with that of $\gamma_0$ up to (and including) order $q-3$ at the north and south pole of the four-sphere. This implies (see, for instance, \cite{Wil93} pp. 90-92) that the expansion of the metric $\gamma$ in geodesic normal coordinates around each of those poles reads 
\[
\gamma_{ij}(w)={\gamma_0}_{ij}(w)+r_{ij}(w), \ \textrm{for} \ r_{ij}(w)=O(|w|^{q})
\]
and this is enough to ensure the validity of the smallness estimates \eqref{eq:est1}, \eqref{eq:est2} for the mean curvature of $\tilde{M}_{\eta_1,\eta_2}$ in $B_{\eta_2}$. The rest of the proof follows closely that of Theorem \ref{thm:main}.
\end{proof}

\section{Concluding remarks}\label{sec:concl}

We shall conclude this article with three remarks:

\begin{enumerate}

\item{Our results, specifically Theorem \ref{thm:main}, Theorem \ref{thm:main2} and Theorem \ref{pro:deform} have been stated and proved for $S^4$ for the sake of definiteness, but do have a rather straighforward extension to the six-dimensional sphere. In fact, our approach seems flexible enough to be easily adapted to all dimensions for which Hsiang desingularizations exist, namely all dealt with in the trilogy \cite{Hsi83a, Hsi83b, HS86}. The necessary changes should mostly be of notational character.}

\item{The general principle that lies behind our construction, that is \textsl{deforming and desingularizing a minimal submanifold} has been successfully developed in the category of special Lagrangian submanifolds, see the works by Joyce (specifically \cite{Joy03} and references therein) and Pacini \cite{Pac13b, Pac13c}. However, the case of (general) minimal submanifolds is much different, as is witnessed by a comparison (even in the smooth setting) of the results of the deformation theory of \cite{Whi91} with those of \cite{McL98}. The problem of developing a perturbation theory for minimal submanifolds with isolated singularities, which shall be effectively applicable to \textsl{some cases of natural geometric interest} seems rather hard. Using the methods of Section \ref{sec:deform} based on computing indicial roots in suitably weighted Sobolev spaces, one can see that the Jacobi operator of the Clifford football $J_{M}:\mathcal{W}^{k,2}_{\boldsymbol{\beta}}\to \mathcal{W}^{k-2,2}_{\boldsymbol{\beta}-2}$ has Fredholm index equal to -18 for any $k\geq 2$ and $\boldsymbol{\beta}=(\beta,\beta), \ \textrm{for} \ \beta>1$. This implies that the natural Lyapunov-Schmidt reduction (see, for instance, chapter 2 of \cite{AM06}) is doomed to fail, for there are only 16 geometric degrees of freedom (corresponding to moving the singularities by local isometries and acting on the horizontal Clifford tori). The two extra elements in the cokernel of $J_M$ obstructing the deformation problem for the Clifford football correspond to \textsl{regularizing modes} associated to the presence of the Hsiang desingularizations themselves.}
\item{Therefore, it is natural to pose the following open problems:
	\begin{enumerate}
	\item{Is it true that for any Riemannian metric $\gamma$ in a suitably small neighbourhood of $\gamma_0$ one can find a perturbed Clifford football in $(S^4,\gamma)$?}
	\item{Are there examples of Riemannian manifolds $(M^n,\gamma)$ that are not round spheres and yet contain infinitely many embedded minimal (closed) hypersurfaces of fixed topology, bounded volume and arbitrarily large Morse index?}
	\item{Is it true that any Riemannian metric $\gamma$ in a suitably small neighbourhood of $\gamma_0$ the Riemannian manifold $(S^4,\gamma)$ contains infinitely many embedded minimal hyperspheres? (If that were the case, this would be in striking contrast with the conclusion of Theorem 4.5 in \cite{Whi91}, where examples are given of almost-round metrics on $S^3$ for which there are exactly four minimal two-spheres).}
	\end{enumerate}	These are truly fascinating questions and we certainly expect them to generate an impressive amount of interesting research for many years to come.}
     
 \end{enumerate}   
     
\bibliographystyle{plain}

\end{document}